\documentclass[11pt]{article}
\usepackage{amsmath,latexsym,amssymb}

\usepackage{tikz}
\usetikzlibrary{arrows.meta}
\usepackage{mathabx}
\usepackage{pdflscape}
\usepackage{rotating}
\usepackage{graphics}
\usepackage{enumerate}
\newif\ifpdf
\pdffalse
\ifpdf

\else

\fi
\providecommand{\1}{\mathbf{1}}
\newcommand{\ignore}[1]{}
\newcommand{\mfalls}{\quad\mbox{if \;}}

\renewcommand{\cases}[1]{\left\{\begin{array}{rl}#1\end{array}\right.}

\newcommand{\mbu}{\quad\mbox{ and }\quad}
\newcommand{\mbs}[1]{\mbox{ \;#1\; }}

\newcommand{\mf}{\quad\mbox{\;for \;}}
\newcommand{\mfa}{\quad\mbox{\;for all \;}}

\renewcommand{\P}{\mathbf{P}}
\newcommand{\E}{\mathbf{E}}
\newcommand{\Var}{\mathbf{Var}}

\def\given{\hspace{0.8pt}|\hspace{0.8pt}}
\def\Given{\hspace{1.0pt}\big|\hspace{1.0pt}}

\newcommand{\R}{{\mathbb{R}}}
\newcommand{\N}{{\mathbb{N}}}

\newcommand{\EE}{\mathbb{E}}
\newcommand{\PP}{\mathbb{P}}

\newcommand{\C}{{\mathbb{C}}}
\newcommand{\Z}{{\mathbb{Z}}}

\newcommand{\ve}{\varepsilon}

\newcommand{\bec}{\begin{equation}}
\newcommand{\eec}{\end{equation}}
\newcommand{\bac}{\begin{eqnarray}}
\newcommand{\eac}{\end{eqnarray}}
\newcommand{\be}{\begin{displaymath}}
\newcommand{\ee}{\end{displaymath}}
\newcommand{\ba}{\begin{eqnarray*}}
\newcommand{\ea}{\end{eqnarray*}}
\newcommand{\nc}{\nonumber \\}
\newcommand{\equ}[1]{(\ref{#1})}

\newtheorem{proposition}{Proposition}[section]
\newtheorem{theorem}[proposition]{Theorem}

\newtheorem{lemma}[proposition]{Lemma}
\newtheorem{corollary}[proposition]{Corollary}
\newtheorem{uremark}[proposition]{Remark}
\newenvironment{remark}{\begin{uremark}\em}{\hfill$\Diamond$\end{uremark}}
\newtheorem{uexample}[proposition]{Example}

\def\qed{\mbox{$\Box$}}
\def\boxp{\mbox{$\Box$}}
\newenvironment{Proof}{\par\noindent{\bf Proof.\ }}{}
\newenvironment{proof}{\par\noindent{\bf Proof.\ }}{\hfill\qed\\ }

\newtheorem{udefi}{Definition}[section]
\newtheorem{utheo}{Theorem}
\newtheorem{usatz}[udefi]{Satz}
\newtheorem{uprop}[udefi]{Proposition}

\newtheorem{ubemerkung}[udefi]{Bemerkung}
\newtheorem{ukorollar}[udefi]{Korollar}

\newenvironment{beweis}{\subsubsection*{Beweis}}{\qed }

\def\Re{\mathrm{Re}}

\def\CL{\mathcal{L}}

\usepackage{color}
\definecolor{Red}{rgb}{1,0,0}

\newcounter{nump}
\newcounter{numb}
\newcounter{numu}
\def\boxp{\addtocounter{nump}{1}\ensuremath{\Box}}

\def\boxu{\addtocounter{numu}{1}\ensuremath{\clubsuit}}
\def\qed{\hfill\phantom{\boxp}\addtocounter{nump}{-1}\hspace*{-1em}\hfill\boxp}

\def\eope{\tag*{\boxp}}

\def\eoue{\tag*{\boxu}\]\end{Ueb}}

\def\bbw{\begin{beweis}}
\newcommand{\bbwc}[1]{\begin{beweis}\textbf{\!\!(#1)}\quad}
\def\ebwn{\end{beweis}}
\def\ebw{\qed\end{beweis}}

\begin{document}
\title{The tail of the length of an excursion in a trap of random size}
\author{
Nina Gantert\\
Fakult\"{a}t f\"{u}r Mathematik\\
Technische Universit\"{a}t M\"{u}nchen\\
Boltzmannstr. 3\\
85748 Garching\\
Germany\\
gantert@ma.tum.de
\and Achim Klenke\\
Institut f\"{u}r Mathematik\\
Johannes Gutenberg-Universit\"{a}t Mainz\\
Staudingerweg 9\\
55099 Mainz\\Germany\\
math@aklenke.de}

\date{\small submitted 04.11.2021\\version of 05.07.2022}

\maketitle
\begin{abstract}
Consider a random walk with a drift to the right on $\{0,\ldots,k\}$ where $k$ is random and geometrically distributed. We show that the tail $\PP[T>t]$ of the length $T$ of an excursion from $0$ decreases up to constants like $t^{-\varrho}$ for some $\varrho>0$ but is not regularly varying. We compute the oscillations of $t^\varrho\,\PP[T>t]$ as $t\to\infty$ explicitly.
\end{abstract}
\section{Introduction and Main Result}
\label{S1}
\subsection{Introduction}
\label{S1.1}
In this paper, we study a simple object: the tail of the time a biased random walk spends in a trap of random size. Our result is very explicit and may serve as a building block in the study of trapping models.
Trapping phenomena for biased walks have been investigated intensively over the last decade, we refer to \cite{GBAFri} for a survey.
As a model for transport in an inhomogeneous medium, one can study biased random walk on a supercritical percolation cluster on $\Z^d$ for $d\geq 2$.
It turns out that for small values of the bias, the walk moves at a positive linear speed, whereas for large values of the bias, the speed vanishes. The critical value of the bias separating the two regimes is the value where the expectation of the time spent in a trap changes from being finite to being infinite.
This model goes back to \cite{BarmaDhar} and was investigated in \cite{BGP} and \cite{SznitPerc}. Finally, Alexander Fribergh and Alan Hammond proved a sharp transition for the positivity of the speed in \cite{FriHam}.
Concerning limit laws for the distribution of the walker, a central limit theorem for small bias was proved in \cite{SznitPerc}.
The law of the walker in the subballistic case was addressed by \cite{FriHam}: the authors find the polynomial order of the distance of the walker to the origin.
It is conjectured that it depends on the spatial direction of the bias if there is a limit law for the distance of the walker to the origin.

Replacing the integer lattice with a tree yields a biased random walk on a supercritical Galton-Watson tree. In this case, the phase transition for the bias is easier to understand and was shown in \cite{LPP}.
It turns out that the distance of the walker to the origin does not satisfy a limit law, but there are subsequences converging to certain infinitely divisible laws, see \cite{BFGH}. The crucial object is the time $T$ spent in traps (averaged over the size of the trap): since the tail of this random variable is not in the domain of attraction of a stable law, there is no limit law for the time the walker needs to go at a certain distance of the origin.
We refer to the introduction of \cite{BFGH} for more explanations. If one randomizes the bias, the situation changes, see \cite{GBAHa} and \cite{Ham}.
For one-dimensional random walk in random environment, limit laws for the distance of the walker to the origin have been proved in \cite{KKS} under a non-lattice assumption. If the non-lattice assumption is violated, one would expect convergence of subsequences as for the aforementioned biased random walk on a Galton-Watson tree.
The result of this paper can be used to confirm this in the simple case of an environment which has either a drift to the left or a reflection to the right, treated in \cite{Solomon1975} and \cite{Ga}.

As a toy model for the supercritical percolation cluster, one may consider a percolation on a ladder graph, conditioned to survive.
This model was introduced in \cite{AxHae} and further investigated by \cite{GMM1, GMM2, MeiLue}.
Again, our result may be applied to show that there is no limit law for the distance of the walker to the origin, as conjectured in \cite{MeiLue}.

There is a well-known connection between hitting times of a random walk (or random walk in random environment) and the total population size in a branching process (or branching process in random environment) with geometric offspring laws.
For subcritical branching processes in random environment (BPRE), a precise asymptotics for the tail of the total population size under a non-lattice assumption was given in \cite{Afa}. See also \cite{DPZ} for an upper bound on the same tail without non-lattice assumption.
Again, our result can serve as an example that the precise asymptotics fails in the lattice case, at least in a particular case of a degenerate environment.
More precisely,
consider a subcritical BPRE where in each generation the law of the offspring is either geometric with expectation $>1$ or the Dirac measure at $0$.
Denote by $T$ the total population size in this BPRE. Then, while the probability $\PP[T > t]$ satisfies, for positive constants $c_1$ and $c_2$ and a certain exponent $\varrho$,
\begin{equation}
\label{E1.01}
c_1t^{-\varrho}\leq \PP[T > t] \leq c_2t^{-\varrho},
\end{equation}
 it is not regularly varying. More precisely, we show that $ \PP[T > t] t^{\varrho}$ is asymptotically equivalent to a nonconstant, multiplicatively periodic function, see \equ{E1.10}.
In our setup, with $T$ denoting the time spent in a trap of random size, \eqref{E1.01} was proved in \cite{MeiLue} and it was conjectured that the tail is not regularly varying. This is confirmed by our result.
Similar tail asymptotics for various quantities are known in the context of branching processes, see for instance
\cite{Vatutin1974}, \cite{BigginsBingham1991}, \cite{BigginsNadarajah1993}.

\subsection{Main result}
\label{S1.2}
Let us now give precise definitions and state our main result, Theorem~\ref{T1.01}.
Let $\beta>1$ be a fixed parameter.
Let $k\in\N_0$ and consider discrete time random walk $X$ on $\{0,\ldots,k\}$ with edge weight $C(l,l+1)=\beta^l$ along the edge $\{l,l+1\}$ and started in $X_0=0$. That is, if $X_n=l\in\{1,\ldots,k-1\}$ then it jumps to $l+1$ with probability $\beta/(1+\beta)$ and to $l-1$ with probability $1/(1+\beta)$. There is reflection at the boundaries: If $X_n=0$, then it jumps to $1$. If $X_n=k$, then it jumps to $k-1$. Of course, for $k=0$, the random walk is trivial. Let $\P_k$ denote the probabilities with respect to fixed $k$ and let $\PP$ denote the probabilities with respect to a random geometrically distributed $k$ with parameter $1-\alpha$, that is,
\begin{equation}
\label{E1.02}
\PP=\sum_{k=0}^\infty(1-\alpha)\alpha^k\,\P_k.
\end{equation}
Also let $\E_k$ and $\EE$ be the corresponding expectations, respectively.
Here $\alpha\in(0,1)$ is a fixed parameter.
Let
\begin{equation}
\label{E1.03}
T:=\inf\big\{t>0:\,X_t=0\big\}\mfalls k\geq1
\end{equation}
and $T=0$ if $k=0$, be the length of an excursion from $0$.
Let \begin{equation}
\label{E1.04}
\varrho:=-\frac{\log(\alpha)}{\log(\beta)}.
\end{equation}
Our random walk $X$ is a special case of a random walk in an irreducible electrical network, see, e.g., \cite[Chapter 19]{Klenke2020}, on a finite graph $(V,E)$ with edge weights $C(e)$, $e\in E$. Denote by $C(x)$ the sum of $C(e)$ for all edges incident to the vertex $x\in V$, and let $C:=\sum_xC(x)$. The transition probabilities are given by $p(x,y)=C(\{x,y\})/C(x)$, $x,y\in V$. It is easy to check that $\pi(x):=C(x)/C$ defines the unique invariant measure. By \cite[Theorem 17.52]{Klenke2020}, the expected time to return to $x$ (when started in $x$) equals $1/\pi(x)=C/C(x)$.

We use this fact to compute, for fixed $k$ the expectation of $T$:
\begin{equation}
\label{E1.05}
\E_k[T]=\frac{2}{C(0,1)}\sum_{l=0}^{k-1}C(l,l+1)=2\sum_{l=0}^{k-1}\beta^l=2\frac{\beta^k-1}{\beta-1}.
\end{equation}
Hence
\begin{equation}
\label{E1.06}
\EE[T]=(1-\alpha)\sum_{k=1}^\infty\alpha^k\,2\,\frac{\beta^k-1}{\beta-1}
=\cases{
\frac{2\alpha}{1-\alpha\beta}<\infty,&\mfalls \varrho>1,\\[2mm]
\infty,&\mfalls\varrho\leq 1.
}
\end{equation}
A similar but more involved computation shows that
\begin{equation}
\label{E1.07}
\EE\big[T^2\big]<\infty\mbs{if and only if}\varrho>2.
\end{equation}

In order to describe the tail of $T$, we introduce the function $g$ defined by
\begin{equation}
\label{E1.08}
g(t):=\frac{\beta-1}{\beta}\frac{(1-\alpha)\Gamma(\varrho)}{\log(\beta)}\,\left(\frac{2\beta}{(\beta-1)^2}\right)^\varrho\left[1+\sum_{\ell=1}^\infty c_\ell\cos\left(2\pi \ell\frac{\log(t)}{\log(\beta)}-d_\ell\right)\right]
\end{equation}
with
\begin{equation}
\label{E1.09}
c_\ell=2\frac{\left|\Gamma\left(\varrho+\frac{2\pi i\,\ell}{\log(\beta)}\right)\right|}{\Gamma(\varrho)}\mbu
d_\ell=\arg\left(\Gamma\left(\varrho+\frac{2\pi i\,\ell}{\log(\beta)}\right)\right).
\end{equation}
Here, $\Gamma$ is Euler's Gamma function and $\arg(a+bi)\in(-\pi/2,\pi/2)$ denotes the angle of $a+bi$ for $a>0$ and $b\in\R$.
Note that the $c_\ell$ decrease quickly with $\ell$ and hence the constant and the $\ell=1$ mode are dominant.

Note that $g$ is a nonconstant multiplicatively periodic function, that is
\begin{equation}
\label{gperiodic}
g(\beta t)=g(t)\mfa t>0.
\end{equation}
In particular, $g$ is not slowly varying.
\begin{theorem}
\label{T1.01}
For $g$ defined in \equ{E1.08}, we have
\begin{equation}
\label{E1.10}
\lim_{t\to\infty}\frac{ t^{\varrho}\,\PP[T>t]}{g\left(\frac{(\beta-1)^2}{2\beta}t\right)}=1\, .
\end{equation}
\end{theorem}
\subsection{Outline}
\label{S1.3}
The strategy of the proof is as follows: We first consider the event $A$ where $X$ reaches $k$ before returning to $0$. On the complement of this event, $T$ is very small and hence this case can be neglected for the tail of $T$ (Lemma~\ref{L2.03}). On the event $A$, we split the time $T$ into three parts:
\begin{itemize}
\item[(1)] the time $T_{\mathrm{in}}$ needed to reach $k$,
\item[(2)] the time $T_{\mathrm{exc}}$ spent in excursions from $k$ to $k$ that do not reach $0$, and
\item[(3)] the length $T_{\mathrm{out}}$ of the last excursion from $k$ to $0$.
\end{itemize}
We will show that the contributions from (1) and (3) can be neglected (Lemma~\ref{L2.04} and Lemma~\ref{L2.05}). Finally, we consider (2). The number of excursions is geometrically distributed and the length of the single excursion has exponential moments. We infer that the tail of $T$ is governed by the number of excursions multiplied by their expected lengths (Proposition~\ref{P2.17}). The number of excursions is geometrically distributed with a parameter that depends on $k$. We use a very detailed analysis to determine the tail averaged over $k$.

\section{Proofs}
\label{S2}
\setcounter{equation}{0}
\subsection{The time to get in and out}
\label{S2.1}

Let
\begin{equation}
\label{E2.01}
T_{\mathrm{in}}:=\inf\big\{t>0:\,X_t=k\big\}
\end{equation}
be the time it takes to hit the right end of the interval. Let
\begin{equation}
\label{E2.02}
T_{\mathrm{last}}:=\sup\big\{t<T:\,X_t=k\big\}
\end{equation}
be the last visit (if any) of the right end of the interval before returning to $0$.
Let
\begin{equation}
\label{E2.03}
T_{\mathrm{out}}:=T-T_{\mathrm{last}}
\end{equation}
denote the time it takes for this last excursion from $k$ to hit $0$. Finally, let
\begin{equation}
\label{E2.04}
T_{\mathrm{exc}}:=T_{\mathrm{last}}-T_{\mathrm{in}},
\end{equation}
denote the time, the random walk spends in excursions from $k$ before the last excursion from $k$ starts. The random times $T_{\mathrm{exc}}$, $T_{\mathrm{last}}$ and $T_{\mathrm{out}}$ are well-defined on the event
\begin{equation}
\label{E2.05}
A:=\big\{T_{\mathrm{in}}<T\}
\end{equation}
In fact, on $A$, we have $T_{\mathrm{last}}<\infty$.

\begin{lemma}
\label{L2.01}
\begin{equation}
\label{E2.06}
\P_k[A]=\frac{\beta-1}{\beta-\beta^{1-k}}\geq\frac{\beta-1}{\beta}.
\end{equation}
\end{lemma}
\begin{Proof}
Considering $\{0,\ldots,k\}$ as an electrical network with resistances $R(l,l+1)=\beta^{-l}$, we get the effective resistances $R_{\mathrm{eff}}(0,1)=1$ and
$$R_{\mathrm{eff}}(0,k)=1+\beta^{-1}+\ldots+\beta^{-k+1}=\frac{1-\beta^{-k}}{1-1/\beta}.$$
Now (compare, e.g., \cite[(19.9)]{Klenke2020})
\[\P_k[A]=\frac{R_{\mathrm{eff}}(0,1)}{R_{\mathrm{eff}}(0,k)}=\frac{\beta-1}{\beta-\beta^{1-k}}.\eope\]
\end{Proof}

On $A^c$, until time $T$, $X$ is a random walk conditioned to return to $0$ before hitting $k$. Now let $U$ be such a random walk started in $U_0=0$. Let $T^{U}:=\inf\big\{t>0:U_t=0\big\}$. Then
\begin{equation}
\label{E2.07}
\P_k\big[T^U=t\big]=\P_k\big[T=t\Given A^c\big]\mfa t.
\end{equation}

The transition probabilities of $U$ can be computed via Doob's $h$-transforms. Let $h_k(l)=\beta^{-l}-\beta^{-k}$ be a harmonic (on $\{1,\ldots,k-1\}$) function for $X$ with $h_k(k)=0$ and $h_k(0)>0$. Then for $l=1,\ldots,k-1$, we have
\begin{equation}
\label{E2.08}
\P_k\big[U_{t+1}=l+1\Given U_t=l\big]=\frac{h_k(l+1)}{h_k(l)}\frac{\beta}{1+\beta}=\frac{1}{\beta+1}\left(1-\frac{\beta-1}{\beta^{k-l}-1}\right).
\end{equation}
We compare $U$ to the random walk $\check Y$ on $\Z$ with conductances $\beta^{-l}$ along the edge $\{l,l+1\}$. That is, $\check Y$ makes a jump to the right with probability $1/(1+\beta)$ and to the left with probability $\beta/(1+\beta)$. Also, let $Y$ be the random walk on $\Z$ with conductances $\beta^l$ along the edge $\{l,l+1\}$. That is, $-Y$ has the same jump probabilities as $\check Y$. Let
$$T^Y:=\inf\big\{t>0:Y_t=0\big\}\mbu T^{\check Y}:=\inf\big\{t>0:\,\check Y_t=0\big\}.$$
Clearly, if $Y_0=\check Y_0=0$, then $T^Y$ and $T^{\check Y}$ have the same distribution.
By \equ{E2.08}, we see that $T^U$ is stochastically bounded by $T^{\check Y}$. More precisely, we have
\begin{equation}
\label{E2.09}
\P_k[T^U>t]\leq \P\big[T^{\check Y}>t\Given \check Y_1=1\big].
\end{equation}
\begin{lemma}
\label{L2.02}
We have\begin{equation}
\label{E2.10}
\E\big[T^{\check Y}\Given \check Y_1=1\big]=\frac{2\beta}{\beta-1},\qquad \Var\big[T^{\check Y}\Given \check Y_1=1\big]=\frac{4\beta(\beta+1)}{(\beta-1)^3}
\end{equation}
and
\begin{equation}
\label{E2.11}
\E\big[e^{\lambda T^{\check Y}}\Given \check Y_1=1\big]=\frac{1}{2}\left(\beta+1-\sqrt{(\beta+1)^2-4\beta e^{2\lambda}}\right) \mfa \lambda<\log\frac{\beta+1}{2\sqrt{\beta}}.
\end{equation}
By symmetry, the statements also hold for $Y$ instead of $\check Y$ conditioned on $Y_1=-1$.
\end{lemma}
\begin{proof}
Define
$$\tau:=\inf\{t\geq 1:\,\check Y_t=1\}.$$
Define the function $\psi$ by
$$\psi(\lambda)=\E\big[e^{\lambda T^{\check Y}}\Given \check Y_0=1\big]
=e^\lambda\,\E\big[e^{\lambda T^{\check Y}}\Given \check Y_1=1\big].$$
Decomposing according to the position of $\check Y$ at time $1$ and using the strong Markov property at time $\tau$ (in the fourth line) yields
$$\begin{aligned}
\psi(\lambda)
&=
e^\lambda\,\E\big[e^{\lambda T^{\check Y}}\1_{\{\check Y_1=0\}}\Given \check Y_0=1\big]+
e^\lambda\,\E\big[e^{\lambda T^{\check Y}}\1_{\{\check Y_1=2\}}\Given \check Y_0=1\big]\\
&=
\frac{\beta}{1+\beta}e^{2\lambda}+
e^\lambda\,\E\big[e^{\lambda (T^{\check Y}-\tau)}\1_{\{\check Y_1=2\}}\Given \check Y_0=1\big]
\E\big[e^{\lambda \tau}\Given \check Y_1=2\big]\\
&=
\frac{\beta}{1+\beta}e^{2\lambda}
+\frac{1}{1+\beta}
e^\lambda\,\E\big[e^{\lambda (T^{\check Y}-\tau)}\Given\check Y_1=2\big]
\E\big[e^{\lambda \tau}\Given \check Y_1=2\big]\\
&=
\frac{\beta}{1+\beta}e^{2\lambda}
+\frac{1}{1+\beta}
e^{\lambda}\E\big[e^{\lambda T^{\check Y}}\Given\check Y_0=1\big]
\E\big[e^{\lambda \tau}\Given \check Y_1=2\big]\\
&=
\frac{\beta}{1+\beta}e^{2\lambda}
+\frac{1}{1+\beta}
\psi(\lambda)^2.
\end{aligned}$$
This quadratic equation has two solutions which at $\lambda=0$ take the values $1$ and $\beta$, respectively. The relevant one takes the value $1$ and is given in \equ{E2.11}.
Taking the derivatives at $\lambda=0$ gives \equ{E2.10}.
\end{proof}

\begin{lemma}
\label{L2.03}
There exists an $\varepsilon>0$ such that
\begin{equation}
\label{E2.12}
\PP[T>t\given A^c]\leq e^{-\varepsilon t},\qquad t\geq 1.
\end{equation}
\end{lemma}
\begin{proof}
This is a direct consequence of \equ{E2.07}, \equ{E2.09} and the existence of exponential moments (Lemma~\ref{L2.02}).
\end{proof}

\begin{lemma}
\label{L2.04}
There exists an $\varepsilon>0$ such that
\begin{equation}
\label{E2.13}
\PP[T_{\mathrm{in}}>t\given A]\leq e^{-\varepsilon t},\qquad t\geq 1.
\end{equation}
\end{lemma}
\begin{Proof}
By Lemma~\ref{L2.01}, it is enough to show
\begin{equation}
\label{E2.14}
\PP[T_{\mathrm{in}}>t]\leq e^{-\varepsilon t},\qquad t\geq 1.
\end{equation}
Note that $X$ and $Y$ can be coupled such that $X_t\geq Y_t$ for all $t\leq T_{\mathrm{in}}$. Hence
$$\P_k[T_{\mathrm{in}}>t]\leq \P[Y_t<k\given Y_0=0].$$
Now $Y_t$ is a sum of $t$ i.i.d. random variables and $\E[Y_t]=\frac{\beta-1}{\beta+1}\,t>0$, hence
by Cram\'er's theorem, there exists an $\varepsilon>0$ such that
$\P[Y_t<k]\leq e^{-\varepsilon t}$ for $t\,\frac{\beta-1}{\beta+1}\geq2k$.
Hence
$$\P_k[T_{\mathrm{in}}>t] \leq e^{-\varepsilon t}\mbox{ for }t\,\frac{\beta-1}{\beta+1}\geq2k.
$$
Concluding, we have
\[
\PP[T_{\mathrm{in}}>t] \leq e^{-\varepsilon t}+(1-\alpha)\sum_{2k> t(\beta-1)/(\beta+1)}\alpha^k
\leq e^{-\varepsilon t}+\alpha^{t(\beta-1)/2(\beta+1)}.
\eope\]
\end{Proof}
\begin{lemma}
\label{L2.05}
There exists an $\varepsilon>0$ such that
\begin{equation}
\label{E2.15}
\PP[T_{\mathrm{out}}>t\given A]\leq e^{-\varepsilon t},\qquad t\geq 1.
\end{equation}
Furthermore, for each $k\in\N$,
\begin{equation}
\label{E2.16}
\E_k\big[T_{\mathrm{out}}\given A\big]\leq k\,\frac{\beta+1}{\beta-1}.
\end{equation}

\end{lemma}
\begin{Proof}
Let $V$ be a random walk on $\{0,\ldots,k\}$ with the same transition probabilities as $U$ (see \equ{E2.08}) but started at $k$. Define $T^V:=\inf\{t>0:V_t=0\}$.

Note that $V$ can be coupled with $\check Y$ (started in $\check Y_0=k$) such that $V_t\leq \check Y_t$ for $t\leq T^V$.
Arguing as in the proof of Lemma~\ref{L2.04}, we get an $\varepsilon>0$ such that
$$\PP\big[T_{\mathrm{out}}>t\given A\big]=\PP\big[T^V>t\big]\leq \P\big[\check Y_t>0\big]\leq e^{-\ve t},\qquad t\geq1.$$
Let $T^{\check Y}_0:=\inf\{t>0:\,{\check Y}_t=0\}$. Note that ${\check Y}$ has a drift $\frac{\beta-1}{\beta+1}$ to the left. Hence, the average time it takes to visit the point left of the starting point is $\frac{\beta+1}{\beta-1}$. Now $T^{\check Y}_0$ is the time it takes to visit the $k$th point left of the staring point. Hence, again by stochastic domination,
\[\E_k\big[T_{\mathrm{out}}|A\big]\leq \E_k\Big[T^{\check Y}_0\Big]=k\frac{\beta+1}{\beta-1}.\eope\]
\end{Proof}

\subsection{The time spent in excursions}
\label{S2.2}
Recall that $T=T_{\mathrm{in}}+T_{\mathrm{exc}}+T_{\mathrm{out}}$. We have dealt with $T_{\mathrm{in}}$ and $T_{\mathrm{out}}$. Now we turn to the time $T_{\mathrm{exc}}$ the random walk $X$ spends in excursions from $k$ before it hits $0$. These excursions are pieces of the random walk conditioned not to hit $0$. Let $N$ denote the number of these excursions if $A$ occurs and $N=0$ on $A^c$. Note that $N$ is geometrically distributed with respect to the conditional probability $\P_k[\;\boldsymbol{\cdot}\;\given A]$.

Our strategy is
\begin{itemize}
\item
to compute the parameter of $N$ (depending on $k$),
\item
to estimate expectation and exponential moments of the lengths of the excursions and
\item
to show that for the tail of $T$, it is good enough to replace the lengths of the excursions by their expected value.
\end{itemize} Hence, the tail of $N$ rules the game, see Proposition~\ref{P2.17}.

Finally, we will compute the tail of $N$ with an involved analysis using Mellin transforms.

Let $\bar X$ be the random walk on $\{0,\ldots,k\}$ started in $\bar X_0=k$. Let
$$\bar T_0:=\inf\big\{t>0:\,\bar X_t=0\big\}$$
and
$$\bar T_k:=\inf\big\{t>0:\,\bar X_t=k\big\}.$$
Let
$$B:=\{\bar T_k<\bar T_0\}=\{\bar X\mbox{ returns to $k$ before hitting $0$}\}.$$
\begin{lemma}
\label{L2.06}
We have
$$\P_k[B]=1-\frac{\beta-1}{\beta^k-1}.$$
\end{lemma}
\begin{proof}
This is similar to the proof of Lemma~\ref{L2.01}.
\end{proof}

Let $\check X$ be the random walk on $\{0,\ldots,k\}$ started at $\check X_0=k$ and conditioned to return to $k$ before hitting $0$. This means the transition probabilities of $\check X$ are given by Doob's $h$-transform with the harmonic function $h_0(l)=1-\beta^{-l}$. Explicitly, we have
\begin{equation}
\label{E2.17}
\P_k\big[\check X_{t+1}=l+1\Given \check X_{t}=l\big]
=\frac{\beta}{\beta+1}\frac{h(l+1)}{h(l)}=\frac{\beta}{\beta+1}\frac{\beta^{l+1}-1}{\beta^{l+1}-\beta}>\frac{\beta}{\beta+1}.
\end{equation}
Let
\begin{equation}
\label{E2.18}
T^{\check X}_k:=\inf\big\{t>0:\check X_t=k\big\}.
\end{equation}
 Then
$$
\P_k\big[T^{\check X}_k= t\big]=\P_k\big[\bar T_k=t\Given B\big].$$

Let $N,T^{(1)},T^{(2)},\ldots$ be independent random variables with respect to $\P_k$ and such that
$N$ is geometrically distributed with parameter $\P_k[B^c]=\frac{\beta-1}{\beta^k-1}$ and
$$\P_k[T^{(i)}=l]=\P_k[\bar T_k=l\given B],\qquad l\in\N_0,\,i=1,2,\ldots.$$
Also let
\begin{equation}
\label{E2.19}
\widetilde T:=\sum_{i=1}^N T^{(i)}.
\end{equation}
\begin{lemma}
\label{L2.07}
We have
$$\P_k\big[\widetilde T=t\big]=\P_k\big[T_{\mathrm{exc}}=t\Given A\big],\qquad t\in\N_0.$$
\end{lemma}
\begin{proof}
This is a simple application of the strong Markov property.
\end{proof}

\begin{lemma}
\label{L2.08}
\begin{equation}
\label{E2.20}
\E_k[N]\leq\E_k[N\given A]=\frac{\P_k[B]}{\P_k[B^c]}=\frac{\beta^k-\beta}{\beta-1}
\end{equation}
and (since $N=0$ if $k=0$)
\begin{equation}
\label{E2.21}
\EE[N]\leq \EE[N\given A]=(1-\alpha)\sum_{k=1}^\infty \alpha^k\,\E_k[N]=\frac{\alpha^2\beta}{1-\alpha\beta}.
\end{equation}
\end{lemma}
\begin{proof}
This is a direct computation.
\end{proof}

While $\widetilde T$ is the quantity we have to study, it is more convenient to get rid of the randomness inherent in the lengths of the excursions and to replace them by their expected value. Hence, as a substitute for $\widetilde T$, we introduce
\begin{equation}
\label{E2.22}
\hat{T}:=N\cdot \E_k\big[T^{(1)}\big].
\end{equation}
In order to show that $\widetilde T$ and $\hat T$ are in fact close, we estimate the exponential moments of $T^{(1)}$ and use Markov's inequality. As a direct computation of the exponential moments is a bit tricky, we make a little detour and use a comparison argument for branching processes. We prepare for this comparison argument with some considerations on the convex ordering of geometric distributions. Note that for the case $\varrho<2$, a simpler estimate based on variances would be good enough for our purposes. In fact, the variances exist for any fixed $k$ and give estimates of order $t^{-2}$ which is good enough compared with the leading order term $t^{-\varrho}$ if $\varrho<2$.

\begin{lemma}
\label{L2.09}
We can define a family $(W_r)_{r\in(0,1]}$ of geometrically distributed random variables with parameters $r$, such that
$$W_r\mbs{and}W_q-W_r\mbs{are independent if} 0<q\leq r\leq1.$$
We have
\begin{equation}
\label{E2.23}
\P[W_q-W_r=k]=\cases{q\left(1-\frac{q}{r}\right)(1-q)^{k-1},&\mfalls k=1,2,\ldots,\\[2mm]
\frac{q}{r}(r-q),&\mfalls k=0.
}
\end{equation}
\end{lemma}
\begin{proof}
Let $(U_n)_{n\in\N_0}$ be i.i.d.{} random variables uniformly distributed on $[0,1]$. Let
$$W_r:=\inf\{n:U_n\leq r\}.$$
It is easy to check that the $(W_r)$ have the desired properties.
\end{proof}

\begin{lemma}
\label{L2.10}
Let $0<q\leq r\leq 1$ and let $W_q$ and $W_r$ be geometrically distributed with parameters $q$ and $r$, respectively.  Let $\varphi:\R\to[0,\infty)$ be a convex function. Then
\begin{equation}
\label{E2.24}
\E\left[\varphi(W_r-\E[W_r])\right]
\leq
\E\left[\varphi(W_q-\E[W_q])\right]
\end{equation}
\end{lemma}
\begin{proof}
By Lemma~\ref{L2.09}, we may and will assume that $W_r$ and $W_q-W_r$ are independent. Hence
$$W_r-\E[W_r]=\E\big[W_r-\E[W_r]\Given W_r\big]=\E\big[W_q-\E[W_q]\Given W_r\big].$$
By Jensen's inequality, we get
\begin{equation}
\label{E2.25}
\begin{aligned}
\E\big[\varphi\big(W_r-\E[W_r]\big)\big]
&=\E\big[\varphi\big(\E[W_q-\E[W_q]\Given W_r]\big)\big]\\
&\leq\E\big[\E\big[\varphi(W_q-\E[W_q])\Given W_r\big]\big]\\
&=\E\big[\varphi(W_q-\E[W_q])\big].
\end{aligned}
\end{equation}
\end{proof}

\begin{corollary}
\label{C2.11}
For $\lambda\in\R$, $\kappa\geq1$ and $0<q\leq r\leq 1$, we have
\begin{equation}
\label{E2.26}
\E\left[e^{\lambda (W_r-\E[W_r])}\kappa^{W_r}\right]
\leq \E\left[e^{\lambda(W_q- \E[W_q])}\kappa^{W_q}\right].
\end{equation}
\end{corollary}
\begin{Proof}
Let
$\varphi(x):=e^{\lambda x}\kappa^{x}$.
Since $\E[W_q]\geq \E[W_r]$, we get by Lemma~\ref{L2.10}
\begin{align}
\E\left[e^{\lambda(W_r-\E[W_r])}\kappa^{W_r}\right]
&=\E\left[\varphi(W_r-\E[W_r])\right]\,\kappa^{\E[W_r]}\nonumber\\
&\leq\E\left[\varphi(W_q-\E[W_q])\right]\,\kappa^{\E[W_q]}\nonumber\\
&=\E\left[e^{\lambda(W_q-\E[W_q])}\kappa^{W_q}\right].
\eope
\end{align}
\end{Proof}

\begin{lemma}
\label{L2.12}
Let $Z^{(1)}$ and $Z^{(2)}$ be two Galton-Watson branching processes with generation dependent offspring laws and $Z^{(1)}_0=Z^{(2)}_0=1$. Let
$$\check Z^{(i)}:=\sum_{n=0}^\infty Z^{(i)}_n,\qquad i=1,2,$$
be the total population sizes. The offspring law of $Z^{(i)}$ in generation $n$ is assumed to be geometric with parameter $p^{i,n}$, $i=1,2$, $n\in\N_0$. We also assume that $p^{1,n}\leq p^{2,n}$ for all $n\in\N_0$ and
$$\E\big[(\check Z^{(1)})^2\big]<\infty.$$
Then, we have
\begin{equation}
\label{E2.27}
\E\big[\check Z^{(2)}\big]\leq \E\big[\check Z^{(1)}\big]<\infty\mbu \Var\big[\check Z^{(2)}\big]\leq \Var\big[\check Z^{(1)}\big]<\infty.
\end{equation}
For all $\lambda\in\R$ with $\E[e^{\lambda \check Z^{(1)}}]<\infty$, we have
\begin{equation}
\label{E2.28}
\E\left[\exp\left(\lambda (\check Z^{(2)}-\E[\check Z^{(2)}])\right)\right]\leq\E\left[\exp\left(\lambda (\check Z^{(1)}-\E[\check Z^{(1)}])\right)\right].
\end{equation}
In particular, for $\lambda\geq0$,
\begin{equation}
\label{E2.29}
\E\left[\exp\left(\lambda \check Z^{(2)}\right)\right]\leq\E\left[\exp\left(\lambda \check Z^{(1)}\right)\right].
\end{equation}
\end{lemma}
\begin{proof}
First assume that
\begin{equation}
\label{E2.30}
p^{1,n}=1\mf n\geq n_0\mbs{for some} n_0.\end{equation}
Hence $\check Z^{(i)}=Z^{(i)}_0+\ldots+Z^{(i)}_{n_0}$, $i=1,2$.
For $n_0=1$, the statement follows from the expectation and variance formula for the geometric distribution. The induction step from $n_0-1$ to $n_0$ is a simple application of Wald's formula and the Blackwell-Girshick formula.
In order to get rid of assumption \equ{E2.30}, take monotone limits.

For the exponential inequalities we proceed similarly. Consider first the case \equ{E2.30} and $n_0=1$. In this case the assertion is a direct consequence of Corollary~\ref{C2.11}. For the induction step, we assume that the statement is true for $n_0-1$ and we show it for $n_0$.
Define
$$\kappa^{(i)}:=\E[\exp(\lambda(Z^{(i)}_2+\ldots +Z^{(i)}_{n_0}-\E[Z^{(i)}_2+\ldots+Z^{(i)}_{n_0}]))\Given Z^{(i)}_1=1].$$
By the induction hypothesis, applied to the branching processes started at time $1$ instead of $0$, we have
$$1\leq \kappa^{(2)}\leq \kappa^{(1)}.$$
By decomposing according to the value of $Z^{(i)}_1$, we infer (again for the processes started at time $0$)
\begin{equation}
\label{E2.31}
\begin{aligned}
\E\left[\exp\left(\lambda\left(\check Z^{(2)}-\E[\check Z^{(2)}]\right)\right)\right]
&=
\E\left[\exp\left(\lambda\left(Z^{(2)}_{1}+\ldots +Z^{(2)}_{n_0}-\E[Z^{(2)}_1+\ldots+Z^{(2)}_{n_0}]\right)\right)\right]\\
&=
\E\left[\exp\left(\lambda\left(Z^{(2)}_{1}-\E[Z^{(2)}_1]\right)\right)(\kappa^{(2)})^{Z^{(2)}_1}\right]\\
&\leq
\E\left[\exp\left(\lambda\left(Z^{(2)}_{1}-\E[Z^{(2)}_1]\right)\right)(\kappa^{(1)})^{Z^{(2)}_1}\right]\\
&\leq
\E\left[\exp\left(\lambda\left(Z^{(1)}_{1}-\E[Z^{(1)}_1]\right)\right)(\kappa^{(1)})^{Z^{(1)}_1}\right]\\
&=
\E\left[\exp\left(\lambda\left(\check Z^{(1)}-\E[\check Z^{(1)}]\right)\right)\right],
\end{aligned}
\end{equation}
where in the fourth line we used Corollary~\ref{C2.11} and the assumption $p^{1,1}\leq p^{2,1}$.
\end{proof}

\begin{lemma}
\label{L2.13}
We have
\begin{equation}
\label{E2.32}
\frac{2\beta}{\beta-1}-\frac{2\beta(\beta+1)}{\beta-1}\,k\,\beta^{-k}\leq \E_k\big[T^{(1)}\big]\leq \frac{2\beta}{\beta-1}\mfa k\geq 2
\end{equation}
and
\begin{equation}
\label{E2.33}
 \Var_k\big[T^{(1)}\big]\leq \frac{4\beta(\beta+1)}{(\beta-1)^3}\mfa k\geq 2.
\end{equation}
Furthermore, there is a $\delta>0$ such that
\begin{equation}
\label{E2.34}
 \E_k\big[e^{\lambda(T^{(1)}-\E_k[T^{(1)}])}\big]\leq 1+\frac{4   \beta(\beta^2+1)}{(\beta-1)^3}\lambda^2 \mfa \lambda\in[-\delta,\delta],\,k\geq 2.
\end{equation}
\end{lemma}
\begin{proof}
Let $Y$ be the random walk on $\Z$ that jumps to the right with probability $\beta/(1+\beta)$ and to the left with probability $1/(1+\beta)$ starting in $k-1$. Let
$$T^Y_l:=\inf\big\{t>0:Y_t=l\big\},\qquad l=0,\ldots,k.$$
Recall $T^{\check X}_k$ from \equ{E2.18}. By the basic connection between the occupation times of excursions of random walks and Galton-Watson processes with geometric offspring distributions, we see that $\frac{1}{2}(T^Y_k+1)$ has the same distribution as $\check Z^{(1)}$ from Lemma~\ref{L2.12} with $p^{1,n}\equiv\frac{\beta}{\beta+1}$. Similarly, using \equ{E2.17}, we see that $\frac{1}{2}T^{\check X}_k$ has the same distribution as $\check Z^{(2)}$ with
$$p^{2,n}=\frac{\beta}{\beta+1}\frac{\beta^{{k-n}+1}-1}{\beta^{{k-n}+1}-\beta}>p^{1,n},\quad n=0,\ldots,k-1.$$
By Lemma~\ref{L2.12} and Lemma~\ref{L2.02}, we infer
\begin{equation}
\label{E2.35}
\E_k\big[T^{(1)}\big]=\E_k\big[T^{\check X}_k\big]\leq 1+\E_k\big[T^Y_k\big]=\frac{2\beta}{\beta-1}
\end{equation}
and
\begin{equation}
\label{E2.36}
\Var_k\big[T^{(1)}\big]=\Var_k\big[T^{\check X}_k\big]\leq \Var\big[T^Y_k\big]=\frac{4\beta(\beta+1)}{(\beta-1)^3}.
\end{equation}
On the other hand,
$$\begin{aligned}
\E_k\big[T^{(1)}\big]
&=1+\E\big[T^Y_k\Given T^Y_k<T^Y_0\big]\geq 1+\E\big[T^Y_k\1_{\{T^Y_k<T^Y_0\}}\big]\\
&= 1+\E\big[T^Y_k\big]-\P_k\big[T^Y_k>T^Y_0\big]\,\E\big[T^Y_k\Given T^Y_k>T^Y_0\big].
\end{aligned}
$$
By Lemma~\ref{L2.05}, we get
$$\E\big[T^Y_0\Given T^Y_k>T^Y_0\big]=\E_k\big[T_{\mathrm{out}}\Given A\big]\leq \frac{\beta+1}{\beta-1}\,k.$$
Using the Markov property and arguing as in Lemma~\ref{L2.05}, we get
$$\E\big[T^Y_k-T^Y_0\Given T^Y_k>T^Y_0\big]=\frac{\beta+1}{\beta-1}\,k.$$
Summing up and using Lemma~\ref{L2.06} to get $\P_k[T^Y_k>T^Y_0]=\frac{\beta-1}{\beta^k-1}$, we have

$$\begin{aligned}
\E_k\big[T^{(1)}\big]\geq \frac{2\beta}{\beta-1}-\frac{\beta-1}{\beta^k-1}\,k\,2\,\frac{\beta+1}{\beta-1}.
\end{aligned}
$$
Now we turn to the proof of \equ{E2.34}.
Again by Lemma~\ref{L2.12} and Lemma~\ref{L2.02}, we get for $\lambda<\log\frac{\beta+1}{2\sqrt{\beta}}$
\begin{equation}
\label{E2.37}
\begin{aligned}
 \E_k\big[e^{\lambda(T^{(1)}-\E_k[T^{(1)}])}\big]
 &\leq
\E\big[e^{\lambda (T^{Y}-\E[T^{Y}])}\big] \\
&=F(\lambda):=\frac{1}{2}\left(\beta+1-\sqrt{(\beta+1)^2-4\beta e^{2\lambda}}\right)\cdot e^{-(2\beta /(\beta-1))\lambda}.
\end{aligned}
 \end{equation}
The first and second derivatives at zero are
$$F'(0)=0\mbu F''(0)=\frac{4\beta(\beta^2+1)}{(\beta-1)^3}.$$
Hence, by Taylor's theorem, there exists a $\delta>0$ such that
\begin{equation}
\label{E2.38}
F(\lambda)\leq F(0)+F''(0)\lambda ^2=1+\frac{4\beta(\beta^2+1)}{(\beta-1)^3}\lambda^2\mfa \lambda\in[-\delta,\delta].
\end{equation}
Combining \equ{E2.37} and \equ{E2.38} gives \equ{E2.34}.
\end{proof}

Recall $\widetilde T$ and $\hat T$ from \equ{E2.19} and \equ{E2.22}, respectively. We now use the exponential moment estimates on $T^{(1)}$ to get that $\widetilde T$ and $\hat T$ are close.
\begin{lemma}
\label{L2.14}
There is a constant $c>0$, such that for all $t>0$, we have
\begin{equation}
\label{E2.39}
\PP\big[|\widetilde T-\hat T|>t\big]\leq c\,t^{-2\varrho}.
\end{equation}
\end{lemma}
\begin{proof}
By Markov's inequality and Lemma~\ref{L2.13}, there are $\delta>0$ and $C<\infty$ such that for $\lambda\in[0,\delta]$,
\begin{equation}
\label{E2.40}
\begin{aligned}
\P_k\big[|\widetilde T-\hat T|> t\Given N\big]
&\leq e^{-\lambda t}\,\E_k\big[\exp(\lambda|\widetilde T-\hat T|)\Given N\big]\\
&= e^{-\lambda t}\,\E_k\big[\exp(\lambda|T^{(1)}-\E_k[T^{(1)}]|)\big]^N\\
&\leq 2e^{-\lambda t}\,(1+C\lambda^2)^N\\
&\leq 2e^{-\lambda t}\,e^{C\lambda^2N}.
\end{aligned}
\end{equation}

We need to make a good choice for $\lambda$ to make this inequality effective.
Recall that $N$ is geometric with parameter $r_k:=\frac{\beta-1}{\beta^k-1}$ under the conditional probability $\P_k[\,\boldsymbol{\cdot}\,\given A]$.
Define
\begin{equation}
\label{E2.41}
\lambda_k:=\sqrt{\frac1C\log\left(\frac{1-r_k/2}{1-r_k}\right)},\qquad k=2,3,\ldots
\end{equation}
Then we have
\begin{equation}
\label{E2.42}
\E_k\big[e^{C\lambda_k^2N}\Given A\big]=\frac{r_k}{1-(1-r_k)e^{C\lambda_k^2}}=\frac{r_k}{1-(1-r_k/2)}=2,\quad k\geq 2,
\end{equation}
and for $l>k$,
\begin{equation}
\label{E2.43}
\E_k\big[e^{C\lambda_l^2N}\Given A\big]\leq \E_k\big[e^{C\lambda_k^2N}\Given A\big]=2.
\end{equation}
Note that
$\frac{\beta-1}{\beta^k-\beta}<\frac12$ for all $k\geq 2$. Hence (using the fact that $\log(1+x)\geq x/2$ for $x\in[0,1/2]$),
\begin{equation}
\label{E2.44}
\lambda_k=\sqrt{\frac1C\log\left(1+\frac12\frac{\beta-1}{\beta^k-\beta}\right)}
\geq\sqrt{\frac{\beta-1}{4C}}\;\beta^{-k/2}\mfa k\geq 2.
\end{equation}
Let $C':=\sqrt{\frac{\beta-1}{4C}}$.
Note that $\lambda_k\downarrow 0$ and let $k_0\in\N$ be large enough such that $\lambda_k<\delta$ for all $k\geq k_0$.

 Then (using Lemma~\ref{L2.18} with $\sqrt\beta$ instead of $\beta$ and hence $2\varrho$ instead of $\varrho$ in the last step) there is a constant $\tilde C<\infty$ such that
\begin{equation}
\label{E2.45}
\begin{aligned}
\PP\big[|\widetilde T-\hat T|> t\big]
&\leq2(1-\alpha)\sum_{k=1}^\infty\alpha^ke^{-\lambda_{k\vee k_0}t}\,\E_k\Big[e^{C\lambda_{k\vee k_0}^2N}\Given A\Big]\\
&\leq 4e^{-\lambda_{k_0} t}
+2(1-\alpha)\sum_{k=k_0+1}^\infty\alpha^ke^{-\lambda_k t}\,\E_k\Big[e^{C\lambda_k^2N}\Given A\Big]\\
&\leq 4e^{-\lambda_{k_0} t}+4(1-\alpha)\sum_{k=k_0+1}^\infty\alpha^ke^{-C'\beta^{-k/2} t}\\
&\leq 4e^{-\lambda_{k_0} t}+\tilde C t^{-2\varrho}.
\end{aligned}\end{equation}
Since $\lambda_{k_0}>0$ is a constant, the claim follows.
\end{proof}

It is still a bit inconvenient to work with $\hat T$ as the expectation of $T^{(1)}$ depends on $k$, though only slightly. The next step is to replace $\E_k[T^{(1)}]$ in the definition of $\hat T$ by its limit $\lim_{k\to\infty}\E_k[T^{(1)}]=\frac{2\beta}{\beta-1}$.
\begin{lemma}
\label{L2.15}
There is a constant $c> 0$, such that for all $t>0$, we have
\begin{equation}
\label{E2.46}
\PP\left[\left|\hat T-N\frac{2\beta}{\beta-1}\right|> t\right]\leq e^{-c\sqrt{t}}.
\end{equation}
\end{lemma}
\begin{proof}
By Lemma~\ref{L2.13}, and by the fact that $\hat T=2N$ if $k=1$, we know that
$$\left|\hat T-N\frac{2\beta}{\beta-1}\right|\leq N\frac{2\beta(\beta+1)}{\beta-1}k\beta^{-k}\mfa k\geq 1.$$
Hence for any $k_0\in\N$,
$$\begin{aligned}
\PP\left[\left|\hat T-N\frac{2\beta}{\beta-1}\right|> t\right]
&\leq\sum_{k=1}^\infty(1-\alpha)\alpha^k\,\P_k\left[N\frac{2\beta(\beta+1)}{\beta-1}k\beta^{-k}>t\right]\\
&\leq \alpha^{k_0+1}+\sum_{k=1}^{k_0}(1-\alpha)\alpha^k\left(1-\frac{\beta-1}{\beta^k-1}\right)^{t\beta^kk^{-1}\frac{\beta-1}{2\beta(\beta+1)}}\\
&\leq \alpha^{k_0+1}+\exp\left(-\frac{(\beta-1)^2}{2\beta(\beta+1)}k_0^{-1}t\right).
\end{aligned}
$$
Now choose $k_0=\sqrt{t}$ to get the result.
\end{proof}

In order to see that the error terms are smaller than the main term, that is the tail of $N$, we need a lower bound for the tail of $N$. Since we give a more detailed analysis later, here we only make a very rough assertion.
\begin{lemma}
\label{L2.16}
There exists a constant $c>0$ such that
$$\PP[N>t]\geq ct^{-\varrho}\mfa t\geq1.$$
\end{lemma}
\begin{Proof}
For $t\in[1,\beta^2]$, the statement holds with $c=\P[N>\beta^2]$. Now assume $t\geq\beta^2$ and let $c=\frac{\beta-1}{\beta}(1-\alpha)e^{-2\beta^2}$.
Let $k\in\N$, $k\geq2$ be such that $\beta^k\leq t\leq \beta^{k+1}$. Then (recall Lemma~\ref{L2.01} and note that $1-x\geq e^{-2x}$ for $x\in[0,1/2]$)
\begin{align}
\PP[N>t]&\geq \frac{\beta-1}{\beta}\,\PP[N>t\given A]\nonumber\\
&\geq \frac{\beta-1}{\beta}(1-\alpha)\alpha^k\left(1-\frac{\beta-1}{\beta^k-1}\right)^t\nonumber\\
&\geq \frac{\beta-1}{\beta}(1-\alpha)\exp\left(-2\frac{\beta-1}{\beta^k-1}\beta^{k+1}\right)\alpha^k\nonumber
\geq c\,\alpha^k\geq c\,t^{-\varrho}.
\eope
\end{align}
\end{Proof}

We summarize the above discussion in the following proposition.
\begin{proposition}
\label{P2.17}
We have
$$\lim_{t\to\infty}\frac{\PP\big[N>\frac{\beta-1}{2\beta}t \big]}{\PP[T>t]}=1.$$
\end{proposition}
\begin{proof}
By Lemma~\ref{L2.15}, the tails of $\frac{2\beta}{\beta-1}N$ and $\hat T$ coincide in our scale, given by
Lemma~\ref{L2.16}.
 By Lemma~\ref{L2.14}, the tails of $\widetilde T$ and $\hat T$ coincide. Finally, by Lemmas~\ref{L2.04}, \ref{L2.05} and \ref{L2.07} the tails of $T$ and $\widetilde T$ coincide.
\end{proof}

\subsection{The tail of a geometric random variable with random parameter}
\label{S2.3}

In order to compute the tail of $N$, it is convenient to replace the geometrically distributed random variable with parameter $\frac{\beta-1}{\beta^k-1}$ by an exponentially distributed random variable $N'$ with parameter $\beta^{-k}$. Note that we neglected the factor $\beta-1$ and we will re-introduce it by a scaling of $t$.
The tail of $N'$ is given by
\begin{equation}
\label{E2.47}
\PP[N'>t]=f(t):=\sum_{k=0}^\infty (1-\alpha)\alpha^k\exp\big(-\beta^{-k} t\big),\qquad t> 0.
\end{equation}
\begin{lemma}
\label{L2.18}
There are constants $0<C_1<C_2<\infty$ such that
$$C_1t^{-\varrho}\leq f(t)\leq C_2t^{-\varrho}\mfa t>1.$$
\end{lemma}
\begin{Proof}
Let $k\in\N_0$ be chosen such that $\beta^{k-1}\leq t<\beta^k$. Recall that $\varrho=-\log(\alpha)/\log(\beta)$. Then
\begin{equation}
\label{E2.48}
f(t)\geq f(\beta^k)\geq (1-\alpha)\alpha^k\,e^{-1}=(1-\alpha)e^{-1}(\beta^k)^{-\varrho}\geq(1-\alpha)e^{-1}\beta^{-\varrho}\,t^{-\varrho}.
\end{equation}Let
$$C_2:=\alpha^{-1}\sum_{k=-\infty}^\infty (1-\alpha)\alpha^k\exp\big(-\beta^{-k} \big).$$
Note that  $f$ is decreasing and hence for $l\in\Z$ and $\beta^{l+1}>t\geq \beta^l$, we have
\begin{align}
f(t)\leq f(\beta^l)&=\sum_{k=0}^\infty (1-\alpha)\alpha^k\exp\big(-\beta^{l-k}\big)\nc
&=\alpha^l\sum_{k=-l}^\infty (1-\alpha)\alpha^k\exp\big(-\beta^{-k}\big)\leq C_2 \alpha^{l+1}\leq C_2\,t^{-\varrho}.
\eope
\end{align}
\end{Proof}

\begin{lemma}
\label{L2.19}
We have
\begin{equation}
\label{E2.49}\lim_{t\to\infty} \frac{\PP[N>t]}{f\big((\beta-1)t\big)}\frac{\beta}{\beta-1}=1.
\end{equation}
\end{lemma}
\begin{Proof}
By Lemma~\ref{L2.16} and Lemma~\ref{L2.18}, all error terms of order $o(t^{-\varrho})$ can be neglected.
We use this first to show that the summands of $f(t)$ with $\beta^k\leq t^{2/3}$ are negligible:
\begin{equation}
\label{E2.50}
\begin{aligned}
(1-\alpha)\sum_{k:\beta^k\leq t^{2/3}} \alpha^k(1-\beta^{-k})^t
&
\leq(1-\alpha)\sum_{k:\beta^k\leq t^{2/3}} \alpha^k\exp\left(-\beta^{-k}t\right)\\
&
\leq(1-\alpha)\sum_{k=0}^\infty \alpha^k\exp\left(-t^{1/3}\right)\\
&
=\exp\left(-t^{1/3}\right).
\end{aligned}
\end{equation}
For $N$ note that
\begin{equation}
\label{E2.51}
\P_k[N>t]\leq\P_k[N>t\given A]=\left(1-\frac{\beta-1}{\beta^{k}-1}\right)^t\leq \exp\left(-\beta^{-k}(\beta-1)t\right).
\end{equation}
We use this first to show as in \equ{E2.50} that the summands of $\P[N>t]$ with $\beta^k\leq t^{2/3}$ are negligible:
\begin{equation}
\label{E2.52}
(1-\alpha)\sum_{k:\beta^k\leq t^{2/3}} \alpha^k\,\P_k[N>t]
\leq \exp\left(-((\beta-1)t)^{1/3}\right).
\end{equation}

Recall from Lemma~\ref{L2.01} that $\P_k[A]=\frac{\beta-1}{\beta-\beta^{1-k}}$. Let $\varepsilon>0$ and choose $t$ large enough such that $\P_k[A]\leq(1+\varepsilon)\frac{\beta-1}{\beta}$ for all $k$ such that $\beta^k>t^{2/3}$. Then
\begin{equation}
\label{E2.53}
\begin{aligned}
\PP[N&>t]
\leq (1-\alpha)\sum_{k:\beta^k> t^{2/3}} \alpha^k\,\P_k[N>t]+\exp\left(-((\beta-1)t)^{1/3}\right)\\
&\leq (1+\varepsilon)\frac{\beta-1}{\beta}(1-\alpha)\sum_{k:\beta^k> t^{2/3}} \alpha^k\,\P_k[N>t\given A]+\exp\left(-((\beta-1)t)^{1/3}\right)\\
&\leq (1+\varepsilon)\frac{\beta-1}{\beta}(1-\alpha)\sum_{k=0}^\infty \alpha^k\,\exp\left(-\beta^{-k}(\beta-1)t\right)+\exp\left(-((\beta-1)t)^{1/3}\right)\\
&=(1+\varepsilon)\frac{\beta-1}{\beta}f\big((\beta-1)t\big)+\exp\left(-((\beta-1)t)^{1/3}\right).
\end{aligned}
\end{equation}
This shows
$$
\limsup_{t\to\infty} \frac{\PP[N>t]}{f\big((\beta-1)t\big)}\frac{\beta}{\beta-1}\leq1.
$$
Now we come to the complementary estimate for the $\liminf$.

Note that $\log(1-x)\geq -x-x^2$ for $x\in[0,1/2]$. For the summands of $\P[N>t]$ with $\beta^k>t^{2/3}$, and for $t\geq \beta^3$, we have $k\geq2$ (thus $\frac{\beta-1}{\beta^k-1}\leq\frac{1}{\beta+1}\leq\frac12$) and hence
$$\begin{aligned}
\log\left(1-\frac{\beta-1}{\beta^{k}-1}\right)&\geq -\frac{\beta-1}{\beta^{k}-1}-\left(\frac{\beta-1}{\beta^{k}-1}\right)^2\\
&=\frac{\beta(\beta^{k-1}-1)^2+(\beta-1)(\beta^{k-2}-1)}{\beta^{2k-2}(\beta^k-1)^2}-\frac{\beta-1}{\beta^{k}}-\beta^{2-2k}\\
&   \geq-\frac{\beta-1}{\beta^{k}}-\beta^{2-2k}.
\end{aligned}$$
We infer for $C= C(\beta)$ large enough and all $t\geq2$,
\begin{align}\nonumber
\PP[N>t]&=\PP[N>t\given A]\,\PP[A]\geq \PP[N>t\given A]\frac{\beta-1}{\beta}\nonumber\\\nonumber
&\geq\frac{\beta-1}{\beta}(1-\alpha)\sum_{k:\beta^k> t^{2/3}} \alpha^k\left(1-\frac{\beta-1}{\beta^{k}-1}\right)^t\\\nonumber
&
\geq\frac{\beta-1}{\beta}(1-\alpha)\sum_{k:\beta^k> t^{2/3}} \alpha^k\exp\left(-\beta^{-k}(\beta-1)t\right)\exp\left(-\beta^{2-2k}t\right)\\\nonumber
&
\geq\frac{\beta-1}{\beta}(1-\alpha)\sum_{k:\beta^k> t^{2/3}} \alpha^k\exp\left(-\beta^{-k}(\beta-1)t\right)\exp\left(-\beta^2t^{-1/3}\right)\\\nonumber
&
\geq\frac{\beta-1}{\beta}\left(1-\beta^2\,t^{-1/3}\right)(1-\alpha)\sum_{k:\beta^k> t^{2/3}} \alpha^k\exp\left(-\beta^{-k}(\beta-1)t\right)\\\nonumber
&
\geq\frac{\beta-1}{\beta}\left(1-\beta^2\,t^{-1/3}\right)\left(f\big((\beta-1)t\big)-\exp\left(-((\beta -1) t)^{1/3}\right)\right)\\\nonumber
&\geq\frac{\beta-1}{\beta}\left(1-C\,t^{-1/3}\right)\,f\big((\beta-1)t\big).\eope
\end{align}
\end{Proof}

\begin{remark}
Our comparison of the tails of $N'$ and $T$ in Lemma~\ref{L2.19} and Proposition~\ref{P2.17} allows to recover a result of Solomon \cite{Solomon1975} which we briefly sketch here.

Let $\nu:=\frac{2\beta}{(1-\beta)^2}$ and let $\psi$ be the Laplace transform of $\nu\,N'$, that is,
$$\psi(\lambda)=\E\left[e^{-\lambda\,\nu N'}\right],\qquad \lambda\geq0.$$
Using $f$ from \equ{E2.47} and partial integration, we get
\begin{equation}
\label{E2.Lap1}
\psi(\lambda)=1-\nu\lambda\int_0^\infty f(t)e^{-\lambda \nu t}\,dt=1-\nu(1-\alpha)\lambda\sum_{k=0}^\infty \frac{(\alpha\beta)^k}{1+\lambda\nu\beta^k}.
\end{equation}
If $\alpha\beta>1$, we get the asymptotics
\begin{equation}
\label{E2.Lap2}
\alpha^{-\ell}\left(1-\psi\big(\lambda\beta^{-\ell}\big)\right)\stackrel{\ell\to\infty}\longrightarrow \nu(1-\alpha)\lambda\sum_{k=-\infty}^\infty \frac{(\alpha\beta)^k}{1+\lambda\nu\beta^k}
\end{equation}
uniformly in $\lambda\in[1,\beta]$.
Now let $\varphi$ be the Laplace transform of $T$, that is $\varphi(\lambda)=\EE[e^{-\lambda T}]$, $\lambda\geq0$.
Using Lemma~\ref{L2.19} and Proposition~\ref{P2.17}, if $\alpha\beta>1$, it is easy to show that
\begin{equation}
\label{E2.Lap3}
\lim_{\lambda\downarrow0}\frac{1-\varphi(\lambda)}{1-\psi(\lambda)}=\frac{\beta-1}{\beta}.
\end{equation}
In fact, assume we have two probability measures $\mu_1$ and $\mu_2$ on $[0,\infty)$ and $\xi\in(0,\infty)$ such that
\begin{equation}
\label{E2.Lap4}
\lim_{t\to\infty}\frac{\mu_1((t,\infty))}{\mu_2((t,\infty))}=\xi.
\end{equation}
Denote by $\CL_1$ and $\CL_2$ the Laplace transforms of $\mu_1$ and $\mu_2$, respectively. Then

\begin{equation}
\label{E2.Lap5}
\frac{1-\CL_1(\lambda)}{1-\CL_2(\lambda)}=\frac{\int_0^\infty\mu_1((t,\infty))e^{-\lambda t}\,dt}{\int_0^\infty\mu_2((t,\infty))e^{-\lambda t}\,dt}
\longrightarrow\xi,\qquad\mbox{as } \lambda\downarrow0
\end{equation}
if $\mu_1$ (and hence $\mu_2$) have infinite first moment, that is, if
\begin{equation}
\label{E2.Lap6}
\int_0^\infty\mu_1((t,\infty))\,dt=\infty.
\end{equation}
Note that the expectation of $N'$ is infinite if and only if $\alpha\beta\geq1$. Summing up, for $\alpha\beta>1$, we have
\begin{equation}
\label{E2.Lap7}
\alpha^{-\ell}\left(1-\varphi\big(\lambda\beta^{-\ell}\big)\right)\stackrel{\ell\to\infty}\longrightarrow \nu(1-\alpha)\frac{\beta-1}{\beta}\lambda\sum_{k=-\infty}^\infty \frac{(\alpha\beta)^k}{1+\lambda\nu\beta^k}
\end{equation}
uniformly in $\lambda\in[1,\beta]$. A similar asymptotics was found already by Solomon \cite[Lemma (2.10)(ii)]{Solomon1975} for a model of random walk in a random environment on $\Z$ with a drift to the left except for geometrically placed reflection points. His asymptotics is the same as ours except for an obvious factor due to the fact that (i) Solomon's ``traps'' have size at least one while ours start at zero and (ii) our random walk has a positive chance to exit the trap without reaching the bottom.

The usual Tauber theorems that would help to infer the tail behaviour of $T$ from the behaviour of its Laplace transform near zero assume regular variation of the tails (and the Laplace transforms) which is not the case here. Solomon's proof uses asymptotic equivalence of the Laplace transform $\psi$ to the Laplace transform $\varphi$ he is interested in, just as we did above. However, this is possible only in the case $\alpha\beta>1$ which Solomon is mainly concerned with. Our approach of comparing the tails of the approximating random variable $N'$ instead of its Laplace transform allows to deal also with the case $\alpha\beta\leq 1$.
\end{remark}

Now we come to determining the asymptotic behavior of $f(t)$ as $t\to\infty$. The following proposition completes the proof of Theorem~\ref{T1.01}.

Let $\Gamma$ denotes Euler's $\Gamma$ function. Recall that $\Gamma(a-bi)=\overline{\Gamma(a+bi)}$ for $a,b\in \R$, where the overline indicates the complex conjugate. Also let $\arg(a+bi)\in(-\pi/2,\pi/2)$ denote the angle of $a+bi$ for $a>0$ and $b\in\R$.
\begin{proposition}
\label{P2.20}
For all $\gamma>\varrho$, as $t\to\infty$, we have
\begin{equation}
\label{E2.54}
f(t)=\frac{(1-\alpha)\Gamma(\varrho)}{\log(\beta)}\,t^{-\varrho}\left[1+\sum_{\ell\in\Z,\,\ell\neq 0}\frac{1}{\Gamma(\varrho)}\,\Gamma\left(\varrho-\frac{2\pi i\,\ell}{\log(\beta)}\right)\exp\left(2\pi i\,\ell\frac{\log(t)}{\log(\beta)}\right)\right]+O(t^{-\gamma})
\end{equation}
or equivalently
\begin{equation}
\label{E2.55}
f(t)=\frac{(1-\alpha)\Gamma(\varrho)}{\log(\beta)}\,t^{-\varrho}\left[1+\sum_{\ell=1}^\infty c_\ell\cos\left(2\pi \ell\frac{\log(t)}{\log(\beta)}-d_\ell\right)\right]+O(t^{-\gamma})
\end{equation}
with
\begin{equation}
\label{E2.56}
c_\ell=2\frac{\left|\Gamma\left(\varrho+\frac{2\pi i\,\ell}{\log(\beta)}\right)\right|}{\Gamma(\varrho)}\mbu
d_\ell=\arg\left(\Gamma\left(\varrho+\frac{2\pi i\,\ell}{\log(\beta)}\right)\right).
\end{equation}
\end{proposition}

\begin{proof}
The proof of \equ{E2.54} uses Mellin transforms and follows the strategy outlined in \cite[Example 12]{FlajoletGourdonDumas1995}. We define the Mellin transform of $f$ by
\begin{equation}
\label{E2.57}
f^*(z):=\int_0^\infty t^{z-1}f(t)\,dt,\qquad z\in\C,\,\Re(z)\in(0,\rho).
\end{equation}
An explicit computation shows that the integral converges for $z$ in the strip $\Re(z)\in(0,\rho)$ and equals
\begin{equation}
\label{E2.58}
f^*(z)=\frac{\Gamma(z)(1- \alpha)}{1-\alpha\beta^z}.
\end{equation}
That is, $f^*$ is holomorphic for $\Re(z)\in(0,\rho)$ and can be uniquely extended to a meromorphic function in $\C$ with poles in the nonpositive integers and in $\chi_\ell:=\varrho+2\pi i\ell/\log(\beta)$, see Figure~\ref{F07.1}. Let
$$\sum_{n=-\infty}^\infty a_{\ell,n}(z-\chi_\ell)^n$$
be the Laurent series of $f^*(z)$ around the singularity at $\chi_\ell$. Then
\begin{equation}
\label{E2.59}
a_{\ell,-1}= -\frac{\Gamma(\chi_\ell)}{\log(\beta)}\,(1-\alpha).
\end{equation}
\begin{figure}[ht]
\centerline{
\begin{tikzpicture}[scale=1.0]
\def\myrad{0.10cm}
\draw[-{Stealth[scale=1]},color=blue, line width = 0.4mm] (0.7,-3.25) -- (0.7, 3.25);
\draw[-{Stealth[scale=1]},color=blue, line width = 0.4mm] (2.8,-3.25) -- (0.7,-3.25);
\draw[-{Stealth[scale=1]},color=blue, line width = 0.4mm] (0.7, 3.25) -- (2.8, 3.25);
\draw[-{Stealth[scale=1]},color=blue, line width = 0.4mm] (2.8, 3.25) -- (2.8,-3.25);
\draw[ -{Stealth[scale=1.5]}, line width=1pt] (-4.80, 0.00 ) -- ( 4.80, 0.00 );
\draw[ -{Stealth[scale=1.5]}, line width=1pt] ( 0.00,-4.80 ) -- ( 0.00, 4.80 );
\draw[color=black!60, line width = 0.2mm, dashed] (1.5,-4.4) -- (1.5,4.4);
\draw (-0.2 , -0.1 ) node[anchor=north]{$0$};
\draw (-1.05 , -0.1 ) node[anchor=north]{$-1$};
\draw (-2.05 , -0.1 ) node[anchor=north]{$-2$};
\draw (-3.05 , -0.1 ) node[anchor=north]{$-3$};
\draw (-4.05 , -0.1 ) node[anchor=north]{$-4$};
\draw (1.8, +0.5 ) node[anchor=north]{$\chi_0$};
\draw (2.05, -1.05) node[anchor=north]{$\chi_{-1}$};
\draw (2.05, -2.35) node[anchor=north]{$\chi_{-2}$};
\draw (2.05, -3.65) node[anchor=north]{$\chi_{-3}$};
\draw (2.05,  1.55) node[anchor=north]{$\chi_1$};
\draw (2.05,  2.85) node[anchor=north]{$\chi_2$};
\draw (2.05,  4.15) node[anchor=north]{$\chi_3$};
\draw (1.7, -0.1 ) node[anchor=north]{$\varrho$};
\draw (3.0, -0.1 ) node[anchor=north]{$\gamma$};
\draw (0.9, -0.1 ) node[anchor=north]{$\eta$};
\draw (0.8, 3.85 ) node[anchor=north]{$\eta+iR_2$};
\draw (3.1, 3.85 ) node[anchor=north]{$\gamma+iR_2$};
\draw (0.8,-3.93 ) node[anchor=south]{$\eta-iR_2$};
\draw (3.1,-3.93 ) node[anchor=south]{$\gamma-iR_2$};
\shadedraw (0,0) circle(\myrad) ;
\shadedraw (-1,0) circle(\myrad);
\shadedraw (-2,0) circle(\myrad);
\shadedraw (-3,0) circle(\myrad);
\shadedraw (-4,0) circle(\myrad);
\shadedraw (1.5, 0) circle(\myrad);
\shadedraw (1.5,-1.3) circle(\myrad);
\shadedraw (1.5, 1.3) circle(\myrad);
\shadedraw (1.5,-2.6) circle(\myrad);
\shadedraw (1.5, 2.6) circle(\myrad);
\shadedraw (1.5,-3.9) circle(\myrad);
\shadedraw (1.5, 3.9) circle(\myrad);
\end{tikzpicture}
}\caption{Complex plane with the singularities of $f^\ast$ and the integration path.}
\label{F07.1}
\end{figure}
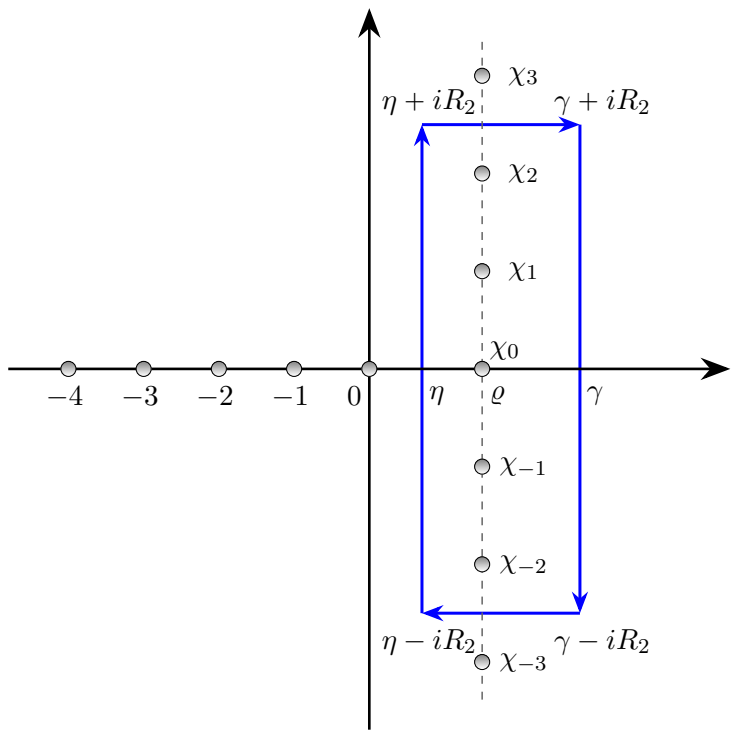
Fix an $\eta\in(0,\varrho)$. The inversion formula for Mellin transforms (see \cite{FlajoletGourdonDumas1995}) gives
\begin{equation}
\label{E2.60}
f(t)=\frac{1}{2\pi i}\int_{\eta-\infty i}^{\eta+\infty i}f^*(z)t^{-z}\,dz.
\end{equation}
Fix some $\gamma>\varrho$. We can approximate the integral by the finite integrals
\begin{equation}
\label{E2.61}
f(t)=\frac{1}{2\pi i}\int_{\eta-R_\ell i}^{\eta+R_\ell i}f^*(z)t^{-z}\,dz
\end{equation}
where $R_\ell=(2\ell+1)\pi/\log(\beta)$. We compute this integral using residue calculus for the path consisting of the four pieces $[\eta-R_\ell i,\eta+R_\ell i]$, $[\eta+R_\ell i,\gamma+R_\ell i]$, $[\gamma+R_\ell i,\gamma-R_\ell i]$ and $[\gamma-R_\ell i,\eta-R_\ell i]$. Note that the horizontal paths do not hit the poles and hence the denominator of $f^*$ is bounded away from $0$ while the modulus of the $\Gamma$ function decreases very quickly with $\ell$. Thus these integrals can be neglected. The integral along the second vertical piece can be estimated by
\begin{equation}
\label{E2.62}
\left|\int_{\gamma-R_\ell i}^{\gamma+R_\ell i}f^*(z)t^{-z}\,dz\right|\leq t^{-\gamma}\frac{1}{\alpha\beta^\gamma-1}\int_{-\infty}^\infty |\Gamma(\gamma +ir)|\,dr.
\end{equation}
As we integrate clockwise, $f(t)$ is minus the sum of the residues in $(\chi_\ell)_{\ell\in\Z}$ plus the $O(t^{-\gamma})$ term. According to \equ{E2.59} these residues are $t^{-\chi_\ell}a_{\ell,-1}=-t^{-\chi_\ell}\Gamma(\chi_\ell)\frac{1-\alpha}{\log(\beta)}$. Concluding, we get \equ{E2.54}.
\end{proof}

Note that while \equ{E2.54} is true for all values of $\gamma$, the constant in the term $O(t^{-\gamma})$ in \equ{E2.54} is of order $\Gamma(\gamma)$, see \equ{E2.62} and thus increases quickly with $\gamma$.

\subsection*{Acknowledgment} We would like to thank our colleague Duco van Straten from Johannes Gutenberg University Mainz for bringing the Mellin transformation to our attention. We would also like to thank the anonymous referees for their extremely careful reading and their very helpful suggestions.


\begin{thebibliography}{10}

\bibitem{Afa}
V.~I. Afanasyev.
\newblock On the maximum of a subcritical branching process in a random
  environment.
\newblock {\em Stochastic Process. Appl.}, 93(1):87--107, 2001.

\bibitem{AxHae}
M.~Axelson-Fisk and O.~H\"{a}ggstr\"{o}m.
\newblock Conditional percolation on one-dimensional lattices.
\newblock {\em Adv. in Appl. Probab.}, 41(4):1102--1122, 2009.

\bibitem{BarmaDhar}
M.~Barma and D.~Dhar.
\newblock Directed diffusion in a percolation network.
\newblock {\em Journal of Physics C}, 16(8), 1983.

\bibitem{GBAFri}
G.~Ben~Arous and A.~Fribergh.
\newblock Biased random walks on random graphs.
\newblock In {\em Probability and statistical physics in {S}t. {P}etersburg},
  volume~91 of {\em Proc. Sympos. Pure Math.}, pages 99--153. Amer. Math. Soc.,
  Providence, RI, 2016.

\bibitem{BFGH}
G.~Ben~Arous, A.~Fribergh, N.~Gantert, and A.~Hammond.
\newblock Biased random walks on {G}alton-{W}atson trees with leaves.
\newblock {\em Ann. Probab.}, 40(1):280--338, 2012.

\bibitem{GBAHa}
G.~Ben~Arous and A.~Hammond.
\newblock Randomly biased walks on subcritical trees.
\newblock {\em Comm. Pure Appl. Math.}, 65(11):1481--1527, 2012.

\bibitem{BGP}
N.~Berger, N.~Gantert, and Y.~Peres.
\newblock The speed of biased random walk on percolation clusters.
\newblock {\em Probab. Theory Related Fields}, 126(2):221--242, 2003.

\bibitem{BigginsBingham1991}
J.~D. Biggins and N.~H. Bingham.
\newblock Near-constancy phenomena in branching processes.
\newblock {\em Math. Proc. Cambridge Philos. Soc.}, 110(3):545--558, 1991.

\bibitem{BigginsNadarajah1993}
J.~D. Biggins and S.~Nadarajah.
\newblock Near-constancy of the {H}arris function in the simple branching
  process.
\newblock {\em Comm. Statist. Stochastic Models}, 9(3):435--444, 1993.

\bibitem{DPZ}
A.~Dembo, Y.~Peres, and O.~Zeitouni.
\newblock Tail estimates for one-dimensional random walk in random environment.
\newblock {\em Comm. Math. Phys.}, 181(3):667--683, 1996.

\bibitem{FlajoletGourdonDumas1995}
Ph.~Flajolet, X.~Gourdon, and Ph.~Dumas.
\newblock Mellin transforms and asymptotics: harmonic sums.
\newblock volume 144, pages 3--58. 1995.
\newblock Special volume on mathematical analysis of algorithms.

\bibitem{FriHam}
A.~Fribergh and A.~Hammond.
\newblock Phase transition for the speed of the biased random walk on the
  supercritical percolation cluster.
\newblock {\em Comm. Pure Appl. Math.}, 67(2):173--245, 2014.

\bibitem{Ga}
N.~Gantert.
\newblock Subexponential tail asymptotics for a random walk with randomly
  placed one-way nodes.
\newblock {\em Ann. Inst. H. Poincar\'{e} Probab. Statist.}, 38(1):1--16, 2002.

\bibitem{GMM1}
N.~Gantert, M.~Meiners, and S.~M\"{u}ller.
\newblock Regularity of the speed of biased random walk in a one-dimensional
  percolation model.
\newblock {\em J. Stat. Phys.}, 170(6):1123--1160, 2018.

\bibitem{GMM2}
N.~Gantert, M.~Meiners, and S.~M\"{u}ller.
\newblock Einstein relation for random walk in a one-dimensional percolation
  model.
\newblock {\em J. Stat. Phys.}, 176(4):737--772, 2019.

\bibitem{Ham}
A.~Hammond.
\newblock Stable limit laws for randomly biased walks on supercritical trees.
\newblock {\em Ann. Probab.}, 41(3A):1694--1766, 2013.

\bibitem{KKS}
H.~Kesten, M.~V.~Kozlov, and F.~Spitzer.
\newblock A limit law for random walk in a random environment.
\newblock {\em Compositio Math.}, 30:145--168, 1975.

\bibitem{Klenke2020}
A.~Klenke.
\newblock {\em Probability theory: {A} comprehensive course.}
\newblock Universitext. Springer, 3rd edition, 2020.

\bibitem{MeiLue}
J.-E.~L\"{u}bbers and M.~Meiners.
\newblock The speed of critically biased random walk in a one-dimensional
  percolation model.
\newblock {\em Electron. J. Probab.}, 24:Paper No. 23, 29, 2019.

\bibitem{LPP}
R.~Lyons, R.~Pemantle, and Y.~Peres.
\newblock Biased random walks on {G}alton-{W}atson trees.
\newblock {\em Probab. Theory Related Fields}, 106(2):249--264, 1996.

\bibitem{Solomon1975}
F.~Solomon.
\newblock Random walks in a random environment.
\newblock {\em Ann. Probability}, 3:1--31, 1975.

\bibitem{SznitPerc}
A.-S.~Sznitman.
\newblock On the anisotropic walk on the supercritical percolation cluster.
\newblock {\em Comm. Math. Phys.}, 240(1-2):123--148, 2003.

\bibitem{Vatutin1974}
V.~A. Vatutin.
\newblock Asymptotic behavior of the probability of the first degeneration for
  branching processes with immigration.
\newblock {\em Teor. Verojatnost. i Primenen.}, 19:26--35, 1974.

\end{thebibliography}
%

\end{document}